\documentclass[12pt, leqno]{article}
\usepackage{amsmath}
\usepackage{amssymb}
\usepackage{theorem}
\usepackage{verbatim}
\usepackage{graphics}
\usepackage[nottoc]{tocbibind}
\usepackage{marvosym}
\usepackage{graphicx}
\usepackage{hyperref}
\usepackage[T1]{fontenc}
\usepackage{tikz}
\usepackage{listings}
\usepackage{ulem}
\usepackage{hyperref}
\usepackage{MnSymbol}

\newcommand{\qed}{\hfill \mbox{\raggedright \rule{.07in}{.1in}}}

\newenvironment{proof}{\vspace{1ex}\noindent{\bf Proof}\hspace{0.5em}} {\hfill\qed\vspace{1ex}}

\usepackage{color}
\usepackage{graphicx}
\usepackage{dsfont}
\usepackage[numbers]{natbib}
\usepackage{framed}
\usepackage{lineno}

\numberwithin{equation}{section} 
\theoremstyle{plain}             

\newtheorem{thm}{Theorem}[section]
\newtheorem{lemma}[thm]{Lemma}
\newtheorem{prop}[thm]{Proposition}
\newtheorem{cor}[thm]{Corollary}

\theoremstyle{definition}    

\newtheorem{defn}[thm]{Definition}
\newtheorem{examp}[thm]{Example}

\newtheorem{rmk}[thm]{Remark}

\newcommand{\C}{\mathds{C}}

\newcommand{\Z}{\mathds{Z}}
\newcommand{\Q}{\mathds{Q}}

\def\addsymbol #1: #2#3{$#1$ \> \parbox{5in}{#2 \dotfill \pageref{#3}}\\}

\input xy
\xyoption{all}

\begin{document}

\title{A proof of the birationality of certain BHK-mirrors. }
\author{Patrick Clarke\\
Drexel University \\
Department of Mathematics}
\date{May 30, 2014}
\maketitle

\abstract{We generalize and give an elementary proof of Kelly's refinement \cite{kelly:2013} 
of Shoemaker's result \cite{shoemaker:2012} on the birationality of certain BHK-mirrors.  Our 
approach uses a construction that is equivalent to the
Krawitz generalization \cite{krawitz:2010} of the duality in Berglund-H\"ubsch \cite{berglund-hubsch:1992}.
 }

\section{Introduction}
We consider certain orbifold quotients of Calabi-Yau hypersurfaces in weighted 
projective spaces.  Such a Calabi-Yau has a mirror partner according to the 
Berglund-H\"ubsch mirror construction \cite{berglund-hubsch:1992}.  This construction was introduced as a 
generalization of Greene-Plesser construction of the mirrors of  Fermat-type 
hypersurfaces \cite{greene-plesser:1990}.  Forming the mirror partner to an orbifold quotient of such a hypersuface is best known as
Berglund-H\"ubsch-Krawitz mirror symmetry since a characterization of the group used to quotient the
mirror hypersurface  was given by Krawitz \cite{krawitz:2010}.

In \cite{shoemaker:2012}, Shoemaker proved that BHK-mirrors of distinct Calabi-Yau orbifolds 
are birational provided the original Calabi-Yau's lie in the same weighted projective space and the group used
to quotient them is the same.   More recently, Kelly 
was able to prove this result assuming only that the group was the same \cite{kelly:2013}; in particular the assumption that the hypersurfaces lie in the same
projective space was dropped.

The present paper proves the result of \cite{kelly:2013} (and a generalization: Theorem \ref{theorem:main}). 
The theorem is found to 
be an easy consequence of some observations about Fourier transforms and transposed matrices.
Our main tool is the duality introduced in 
\cite{clarke:2008}; this duality is a simple construction, and includes
the constructions of \cite{batyrev:1994,borisov:1993, givental:1996, hori-vafa:2000} as special cases.

BHK-mirror symmetry was shown in \cite{clarke:2012} to be a special case of the duality of  \cite{clarke:2008} as well. 
However, since \cite{clarke:2012} was never published,
we include an appendix proving this.

\section{Preliminaries}
Matrices multiply from the left. 
We make the notations
$$
W = \sum_{j=0}^m \prod_{i=0}^n X_i^{p_{ij}} \   \colon \ \C^{n+1} \to \C,
$$
$$
P = (p_{ij})_{ij} \  \colon \ \Z^{m+1} \to \Z^{n+1},
$$
and
$$
\begin{array}{rccc}
(+)  \colon & \C^{m+1} & \longrightarrow & \C \\
& (Z_0, \dotsc, Z_m) & \mapsto &  Z_0 + \dotsc + Z_m
\end{array} 
$$
We assume that the monomials in $W$ are distinct.
\begin{defn}
We consider abelian groups which are isomorphic to finite direct sums whose summands are algebraic tori and/or finitely generated.
For such an abelian group $H$ we have {\bf Fourier transform} and {\bf transposition functors}:
$$
F(H) = \operatorname{Hom}(H, \C^\times),
$$
$$
H^\tau = \operatorname{Hom}(H, \Z).
$$
Homomorphisms are assumed to be algebraic (i.e. $\operatorname{Hom}(\C^\times, \C^\times) = \Z$). 
\end{defn}

\begin{examp}
\label{example:notations}
$W = X_0^{p_{00}}X_1 + X_1^{p_{11}}X_2$ then 
$$
P = \left( \begin{array}{cc} p_{00}& 0\\ 1&p_{11} \\ 0 & 1  \end{array} \right).
$$
$F(\Z) = \C^\times$
and if we denote $Z_i \colon F(\Z^{2}) \to \C$ the map which sends $\chi \mapsto \chi(e_i),$ then 
$F(P)\colon (\C^\times)^3 _{\underline{X}} \to (\C^\times)^2_{\underline{Z}}$ is given by 
$$(Z_0, \ Z_1) =  ( X_0^{p_{00}}X_1, \  X_1^{p_{11}}X_2).$$
\end{examp}

\begin{prop}\  \\
\begin{itemize}
\item For finite $H,$ $F(H) \cong H$ (non-canonically),
\item $F$ is contravariant, and $F \circ F = \mathds{1},$ 
\item $F$ is exact when restricted to finitely generated groups or to affine groups, and it interchanges these two subcategories, 
\item$F(\Z^{\ell+1}) = (\C^\times)^{\ell+1}$ and so $F((\C^\times)^{\ell+1}) = \Z^{\ell+1},$ and
\item $W$ restricted to  $(\C^\times)^{n+1}$ is $(+) \circ F(P).$ 
\end{itemize}
\end{prop}
\begin{proof}The first 4 bullets are usual statements about Pontryagin duality with $S^1$ replaced with $\C^\times,$ and are easily checked.  The last bullet is proved by elaborating on Example \ref{example:notations}. \end{proof}

\begin{figure}[h]
$$
\xymatrix{
1 \ar[r] & G \ar[r]&  (\C^\times)^{n+1} \ar[r] \ar[d]_{F(P)} & (\C^\times)^{n+1}/G \ar@/^/[ddl]^{W_G} \ar[r]  \ar@{-->}[dl]& 1 \\
&  	& (\C^\times)^{m+1} \ar[d]_{(+)} & & \\
& 	& \C & & 
}
$$
\caption{ $G \leq (\C^\times)^{n+1}$ is a closed subgroup such that 
$W$ descends to the quotient.}
\label{figure:descend}
\end{figure}

\newpage
\begin{prop}
If $W$ factors through a quotient $(\C^\times)^{n+1}/G$ of $(\C^\times)^{n+1}$  by a closed subgroup $G$, then  there is a unique homomorphism given by the dashed arrow in Figure \ref{figure:descend} which makes the diagram commute.
\end{prop}
\begin{proof}
Denote by $R_G$ the Reynolds operator for the action of $G$ on the functions on $(\C^\times)^{n+1}.$
$R_G$ is a projection and characters are eigenfunctions for the action of $G$.  Characters are also a basis for the space of functions, so  $R_G$ evaluated on a function is given by expanding the function in characters and erasing those which are not invariant under $G$.  Furthermore, a function descends to $(\C^\times)^{n+1}/G$ 
if and only if it is sent to itself under $R_G.$  This means that the characters in the expansion of $W$ must be $G$-invariant, and thus the map exists.  Since we have assumed the monomials are distinct, it is unique. \end{proof}

\begin{cor}
\label{corollary:factors-factor}
Quotients of $\C^{n+1}$ though which  $W$ factors correspond to factorizations
$$
P = A  \circ B^\tau
$$
where $B^\tau \colon \Z^{m+1} \to M,$  $A \colon M \to \Z^{n+1},$ and $M = F((\C^\times)^{n+1}/G).$
Furthermore, the function to which $W$ descends is 
$$W_G = (+) \circ F(B^\tau) \colon (\C^\times)^{n+1}/G \to \C$$
\end{cor}
\begin{proof}  Apply $F$ to the groups in figure \ref{figure:descend}.  $A = F(\text{quotient map})$ and $B^\tau = F(\text{dotted arrow}).$ \end{proof}

\begin{defn}
We call $([\C^{n+1}/G], W_G)$ a {\bf quotient Landau-Ginzburg model}.  We
assume that such an object includes the data of the presentation: $(\C^{n+1}, W, G).$ 
\end{defn}

\begin{defn}
Given a quotient Landau-Ginzburg model $([\C^{n+1}/G], W_G),$ we define the dual 
quotient Landau-Ginzburg model $([\C^{m+1}/G^T], W^T_{G^T})$ by the data
\begin{itemize}
\item $P^T = P^\tau$, and
\item $G^T = \operatorname{ker} F(B) \leq F((\Z^{m+1})^\tau).$
\end{itemize}
\end{defn}
This definition is illustrated in figure \ref{figure:dual-data}. 
\begin{figure}[h]
$$
\xymatrix{
F(G) & \ar[l]\Z^{n+1} & \ar[l]_{A}  M \\
& \ar[u]^P\Z^{m+1} \ar[ur]_{B^\tau}& 
} 
\qquad \qquad 
\xymatrix{
 F(G^T) & \ar[l](\Z^{m+1})^\tau & \ar[l]_{B}  M^\tau \\
 & \ar[u]^{P^\tau}(\Z^{n+1})^\tau \ar[ur]_{A^\tau}& \\
} 
$$
\caption{The data of a quotient Landau-Ginzburg model and its dual side-by-side.}
\label{figure:dual-data}
\end{figure}

\begin{rmk}
This is the special case of the duality in  \cite{clarke:2008} where the K\"ahler parameter is set to 0, or equivalently all the coefficients of the superpotential are equal to 1.
The point of view in loc. cit. is that the dual is obtained by interchanging the roles of the homomorphisms $A$ and $B$.
\end{rmk}

\section{The theorems of Shoemaker and Kelly}

\begin{prop}
\label{proposition:equal-sups}
If quotient Landau-Ginzburg models $([\C^{n+1}/G], W_G)$ and $([\C^{n+1}/G], W'_{G})$ are quotients by the same group $G$, then $W^T_{G^T}$ equals ${W'}^T_{{G}^{T}}$ on the torus $F(M^\tau).$ 
\end{prop}
\begin{proof}  
Since $G$ is the same in both cases, $A = A'$ in the factorization of $P$ and $P'.$
The definition of the dual and corollary  \ref{corollary:factors-factor} imply  $W^T_{G^T} = {W'}^T_{{G}^T}= (+) \circ F(A^\tau).$
\end{proof}

\begin{defn}
Given a quotient Landau-Ginzburg model $([\C^{n+1}/G], W_G),$ {\bf a weight vector}  is a homomorphism $q \colon M \to \Z$
such that $q B^\tau (e_j -e_k) = 0$ for all $j,k$.
\end{defn}

\begin{defn}
Given a weight vector $q$, $W_G = 0$ defines a hypersurface $Z^q_{W, G}$ in the quotient of the quotient $[(\C^\times)^{n+1}/G/ \C^\times];$ where the action of $\C^\times$ is through the homomorphism  $F(q)(\C^\times)\to  (\C^\times)^{n+1}/G.$  $Z^q_{W, G}$ is called {\bf the associated non-linear sigma-model}. 
\end{defn}

\begin{rmk}
The associated non-linear sigma model  defined here is likely wrong for the purposes of mirror symmetry unless $W$ is a Calabi-Yau polynomial (defined below).  
A more likely candidate, when it exists, is given in Clarke \cite{clarke:2008}.
\end{rmk}

\begin{thm}
\label{theorem:main}
If quotient Landau-Ginzburg models $([\C^{n+1}/G], W_G)$ and $([\C^{n+1}/G], W'_{G})$  are quotients by the same group $G,$ then the weight vectors of 
$([\C^{m+1}/G^T], W_{G^T})$ and $([\C^{m'+1}/{G}^{T'}], W'_{{G}^T})$
 coincide, and $Z^q_{W^T, G^T}$ is birational to $Z^q_{W'^T, G^{T'}}$ for each  $q$.
\end{thm}
\begin{proof}  As before, $A = A'.$  
 The set of weight vectors for the duals are those $q$ which satisfy $qA^\tau(e_j-e_k) = 0$ for all $j,k$, and we know that 
$W_{G^T} = W'_{{G}^{T'}}$ on $F(M^\tau)$ by Proposition \ref{proposition:equal-sups}.
\end{proof}

\begin{defn}
$W$ is called {\bf Calabi-Yau} if  $P$ is square and invertible over $\Q,$
$$(1, \dots, 1) \operatorname{ag}(P) > 0 \text{ or } (1, \dots, 1) \operatorname{ag}(P) < 0$$
and
$$(1, \dots, 1) \operatorname{ag}(P)(1, \dotsc ,1)^\tau= \det(P).$$
Where $\operatorname{ag}(P)$ is the adjugate matrix of $P$.
\end{defn}

We recover the the main theorems of \cite{shoemaker:2012} and \cite{kelly:2013} here.
\begin{cor} 
Consider a group $G \leq (\C^\times)^{n+1}$ and
Calabi-Yau polynomials $W$ and $W'$ which are invariant under the action of $G$.  
Then weight vectors of the duals coincide and form cyclic subgroup of $\operatorname{Hom}(M^\tau, \mathds{Z})$, and 
$Z^{q_0}_{W^T, G^T}$ and $Z^{q_0}_{W'^T, G^{T'}}$ are birational for a  generator $q_0.$
\end{cor}
\begin{proof} Again $A = A'$, 
so the set of weight vectors is cut out by the $n$-equations $qA^\tau(e_0 -e_j) = 0$ for $j = 1, \dotsc, n$.  
These equations are linearly independent because $A$ is invertible, so  these equations are linearly independent and
the weight vectors form a group isomorphic to $\mathds{Z}$.  Now  apply
the above results.
\end{proof}

\newpage
\appendix

\begin{center}{\large APPENDIX}\end{center}

\section{BHK  as Duality for toric LG models}

Berglund-H\"ubsch  \cite{berglund-hubsch:1992} introduced a mirror construction for certain hypersurfaces
in weighted projective spaces given by Calabi-Yau polynomials \cite{berglund-hubsch:1992}.  This was generalized by Krawitz \cite{krawitz:2010} where, to a group $G \leq (\mathds{C}^\times)^N$ 
and Calabi-Yau polynomial $W,$
a dual group
$G^\dagger$ is introduced and  
$$
([(\C^\times)^{n+1}/G],W_G) \text{ and } ([(\C^\times)^{m+1}/G^\dagger],W^T_{G^\dagger}) 
$$
are deemed
dual.
The construction of \cite{berglund-hubsch:1992} occurs when $G = \{1\}.$
We show that when the BHK-dual is defined (i.e. $W$ is Calabi-Yau), it is the same as the one above.

\subsection{Krawitz Duality.}

The notation in this paper is a departure from that of Krawitz \cite{krawitz:2010}, so before proving the equivalence of the constructions, we point out some of the main objects under consideration.

In our construction, given a subgroup $H$ of a torus $T,$ we applied the functor $F$ to obtain an abelian group $F(H)$.
Since $H \leq T,$ we can associate another abelian group to $H;$  namely, those elements of the Lie algebra $\frak{t}$
which map to $H$  under the exponential map $\frak{t} \to T.$ 
 We will denote this subgroup by $\Lambda_H.$  
 
 \begin{defn}
 The assignment $H \mapsto \Lambda_H$ defines a covariant functor $\Lambda_{- }$ from the category 
 of subgroups of tori 
to the category of  abelian groups.
 \end{defn}
 
 To put ourselves in a position that is amenable to the conventions of  \cite{krawitz:2010}, in the case $T = (\mathds{C}^\times)^N$ we denote 
$\Lambda_T$ by $\mathds{C}^N, $ the exponential map will  be $\exp(2\pi i -),$ and so $\Lambda_{\{ 1\}} = \mathds{Z}^N.$

 In \cite{krawitz:2010}   the notation $A_W$ is used for our  $P^\tau,$ the matrix $P$ square and invertible, and $N$ is written for our $n+1$ and $m+1.$
Additionally:
\begin{itemize}
\item $\rho_1, \dotsc, \rho_N =$ the columns of $(P^\tau)^{-1},$ 
\item $\overline{\rho}_1, \dotsc, \overline{\rho}_N =$ the rows of $(P^\tau)^{-1},$ 
\item  $\rho_k$ is also  denote $\exp(2\pi i \rho_k) \in (\mathds{C}^\times)^{N},$ where $\exp(2\pi i -)$ is applied entry-wise to   $\rho_k,$ and
\item similarly,   $\overline{\rho}_k$ is also used to denote $\exp(2\pi i \overline{\rho}_k) \in (\overline{\mathds{C}}^\times)^{N}$
where we have put the $\overline{\ }$ to distinguish between the  $N$-tori.
\end{itemize}
For clarity, we will 
write $g_k = \exp(2\pi i \rho_k)$ and $\overline{g}_k = \exp(2\pi i \overline{\rho}_k)  \in (\mathds{C}^\times)^{N}$  rather than overloading the notations $\rho_k$
and $\overline{\rho}_k.$

Krawitz expresses a subgroup $G$ of $(\mathds{C}^\times)^{N}$ 
under whose action $W$ is invariant in terms of the elements  $\{ g_1, \dotsc, g_N \}.$
This subgroup is specified by which exponents $[a_1, \dotsc, a_N] \in \mathds{Z}^N$ to produce an element $g_1^{a_1} \dotsm g_N^{a_N} \in G.$  Observe that this set of $a$-vectors is exactly 
$\Lambda_G.$  

The group $\langle g_1, \dotsc, g_N \rangle$ was denoted $G_{\bf max}$ in \cite{krawitz:2010}, and (tautologically) we see that  the $\mathds{Z}^N$ in which the $a$-vectors live is 
$\Lambda_{G_{\bf max}}.$  The other  $\mathds{Z}^N$ is $\Lambda_{\{1\}},$
and the vectors $[a_1, \dotsc, a_N]$ which have 
$g_1^{a_1} \cdots  g_N^{a_N} = 1 \in (\mathds{C}^\times)^N$
can be thought of as the columns of $P^\tau$
(because $(P^\tau)^{-1}P^\tau =$ the identity matrix).
This can be packaged into the inclusions
$$
\mathds{Z}^N \cong \Lambda_{\{1\}} \hookrightarrow \Lambda_G \hookrightarrow \Lambda_{G_{\bf max}}\cong \mathds{Z}^N
$$
whose  composition is given by multiplication on the left by  $P^\tau.$

The dual group $G^\dagger$ is defined to be the elements $\overline{g}_1^{r_1} \dotsm \overline{g}_N^{r_N} \in (\overline{\mathds{C}}^\times)^N$ satisfying the condition (18) of \cite{krawitz:2010}:
$$
\left[ \begin{array}{c} r_1, \dotsc, r_N \end{array} \right] (P^\tau)^{-1} \left[ \begin{array}{c} a_1 \\ \vdots \\ a_N \end{array}\right] \in \mathds{Z}
$$
for all $a_1, \cdots, a_N$ such that $g_1^{a_1} \cdots  g_N^{a_N} \in G$.
This means  $\left[ \begin{array}{c} r_1 , \dotsc,  r_N \end{array}\right]$ is in the $\mathds{Z}$-dual space to 
$\Lambda_G$ when this dual space is thought of as a subspace of $\Lambda_{\{1\}}^\tau.$  So if we write 
$\overline{G}_{\bf max}$ for the group $\langle \overline{g}_1, \dotsc, \overline{g}_N  \rangle,$ and 
$\overline{1}$ for the identity element in $(\overline{\mathds{C}}^\times)^N$ we have identifications
$$
\begin{array}{ccccc}
 \Lambda_{G_{\bf max}}^\tau &  \hookrightarrow & \Lambda_G^\tau & \hookrightarrow &   \Lambda_{\{1\}}^\tau \\
 \| &&\|&&\| \\
  \Lambda_{{\{ \overline{1}\}}} &  \hookrightarrow & \Lambda_{G^\dagger} & \hookrightarrow &   \Lambda_{\overline{G}_{\bf max}}.
  \end{array}
$$
and the composition is given by $P.$

\subsection{$G^\dagger = G^T.$}

The basic observation relating the two constructions is the precise relationship between 
the functors $F$, $-^\tau,$ and $\Lambda.$

\begin{lemma}
\label{lemma:F-transpose-exp-kernel}
Considered as functors on algebraic tori, there is a natural isomorphism
$$\lambda \colon F( - )^\tau \Rightarrow \Lambda_{\{1 \in (-) \}}.$$
\end{lemma}
\begin{proof}
Consider an algebraic torus $T.$  $\Lambda_{\{1\}}$ is the kernel of its exponential map.  In particular, we can write $T = \Lambda_{\{1\}} \otimes_\Z (\C/2 \pi i \Z).$ Given $\phi  \colon \Lambda_{\{1\}} \to \Z$ considered as an element of $\Lambda_{\{1\}}^\tau,$ there is  a homomorphism $\chi \colon T \to \C^\times$ by the rule $(k \otimes \overline{z}) \mapsto \phi(k) \cdot \overline{z} \in \C/2 \pi i \Z \cong \C^\times.$  This defines a  isomorphism  $\Lambda_{\{1\}}^\tau \to F(T).$  Transposing this gives us $\lambda  \colon F(-)^\tau \Rightarrow \Lambda_{\{1 \in (-)\}}.$
\end{proof}

\begin{cor}
\label{cor:quotient-map}
Given a finite group $G \leq T$ of a torus, there is natural surjection
$$
q \colon T \to F(\Lambda_{G}^\tau)
$$
with kernel $G.$
\end{cor}
\begin{proof}
Consider the quotient map $T \to T/G.$  Since $G$ is finite, this induces an isomorphism $\Lambda_{G\leq T} \to \Lambda_{1 \in T/G}.$  Now transpose and apply $F$ to Lemma \ref{lemma:F-transpose-exp-kernel}.
\end{proof}

\begin{thm}
Consider a morphism $W \colon \mathds{C}^N \to \mathds{C}$  and a group $G \leq (\mathds{C}^\times)^N$ for which $W$ is invariant.  Assume the matrix $P$ is invertible, then the Krawitz dual $G^\dagger$
equals $G^T.$
\end{thm}
\begin{proof}
From Corollary \ref{cor:quotient-map}, we have an exact sequence
$$
1 \to G^\dagger \to (\overline{\mathds{C}}^\times)^N \to F(\Lambda_G) \to 1
$$ 
Applying $F$ yields
$$
\Lambda_G \to \mathds{Z}^N \to F(G^\dagger).
$$
Thus if we can verify that the map $F(q) \colon \Lambda_G \to \mathds{Z}^N$
equals the map $B \colon M^\tau \to (\mathds{Z}^{m+1})^\tau$ from figure \ref{figure:dual-data}
we are done. Since both $F(q)$ and $B$ factor the invertible matrix $P^\tau,$
if suffices to check that $A^\tau \colon (\mathds{Z}^{n+1})^\tau \to M^\tau$
equals the map $\mathds{Z}^N \to \Lambda_{G}.$  By definition, $A^\tau$ is obtained by
applying $F$ to the map 
$$
(\mathds{C}^\times)^N \to (\mathds{C}^\times)^N/G.
$$
and then  transposing.
On the other hand,  Corollary   \ref{cor:quotient-map}
guarantees that $\mathds{Z}^N = \Lambda_{\{1\}} \to \Lambda_{G}$ is also obtained this way. 
\end{proof}

\subsection*{Acknowledgements} 
The author is grateful for the encouragement of Y. Ruan
and for conversations with M. Krawitz, T. Kelly, and M. Shoemaker.

\newpage

\bibliographystyle{habbrv}
\bibliography{BHK-birat}

\end{document}